\documentclass{amsart}

\newtheorem{theorem}[equation]{Theorem}
\newtheorem{lemma}[equation]{Lemma}
\newtheorem{corollary}[equation]{Corollary}
\newtheorem{proposition}[equation]{Proposition}

\newcommand{\A}{{\mathbb A}}
\newcommand{\C}{{\mathbb C}}
\newcommand{\D}{{\mathbb D}}
\newcommand{\E}{{\mathbb E}}
\newcommand{\X}{{\mathbb X}}

\newcommand{\sD}{{\mathcal D}}
\newcommand{\sO}{{\mathcal O}}

\newcommand{\rank}{{\rm rank}}

\numberwithin{equation}{section}

\usepackage{amscd,amsmath,amssymb,verbatim}

\begin{document}

\title[Newton polytopes and algebraic hypergeometric series]{Newton polytopes and algebraic \\ hypergeometric series}
\author{Alan Adolphson}
\address{Department of Mathematics\\
Oklahoma State University\\
Stillwater, Oklahoma 74078}
\email{adolphs@math.okstate.edu}
\author{Steven Sperber}
\address{School of Mathematics\\
University of Minnesota\\
Minneapolis, Minnesota 55455}
\email{sperber@math.umn.edu}
\address{Department of Mathematics, Princeton University, Princeton, NJ 08544-1000}
\email{nmk@math.princeton.edu}
\date{\today}
\dedicatory{\rm \small with an appendix by Nicholas M. Katz}
\begin{abstract}
Let $X$ be the family of hypersurfaces in the odd-dimen\-sional torus ${\mathbb T}^{2n+1}$ defined by a Laurent polynomial $f$ with fixed exponents and variable coefficients.  We show that if $n\Delta$, the dilation of the Newton polytope $\Delta$ of $f$ by the factor $n$, contains no interior lattice points, then the Picard-Fuchs equation of  $W_{2n}H_{\rm DR}^{2n}(X)$ has a full set of algebraic solutions (where $W_\bullet$ denotes the weight filtration on de Rham cohomology).  We also describe a procedure for finding solutions of these Picard-Fuchs equations.
\end{abstract}
\maketitle

\section{Introduction}

It is a problem of classical interest to determine which hypergeometric series are algebraic functions.  Schwarz\cite{S}  answered this question for the series $_2F_1$, the answer for $_nF_{n-1}$ was given by Beukers-Heckman\cite{BH}.  We refer the reader to Baldassarri-Dwork\cite{BD} for additional results and historical information.  In another direction, Katz\cite{K} showed that for certain differential equations arising from geometry a full set of solutions modulo $p$ for almost all primes $p$ implies a full set of algebraic solutions.  Beukers\cite{Be} applied this result to give a condition for certain nonresonant $A$-hypergeometric series to be algebraic.  

Our interest in this area was stimulated by a result of Rodriguez-Villegas\cite{RV}, who showed that certain series whose coefficients are ratios of factorials are algebraic functions.  Let $\alpha_1,\dots,\alpha_n,\beta_{n+1},\dots,\beta_m$ be a sequence of positive integers satisfying 
\[ \sum_{i=1}^n \alpha_i = \sum_{j=n+1}^m \beta_j, \]
put
\begin{equation}
u_k(\alpha,\beta) =\frac{(\alpha_1k)!\cdots(\alpha_nk)!}{(\beta_{n+1}k)!\cdots(\beta_mk)!},
\end{equation}
and consider the series
\begin{equation}
u(\alpha,\beta;t) = \sum_{k=0}^\infty u_k(\alpha,\beta) t^k.
\end{equation} 
Rodriguez-Villegas showed that (excluding the trivial case where the $\beta_j$ are a permutation of the $\alpha_i$) the series $u(\alpha,\beta;t)$ is an algebraic function if and only if $m=2n+1$ and $u(\alpha,\beta;t)$ has integral coefficients.

From the $A$-hypergeometric point of view these series are resonant so the result of Beukers does not apply directly.  However these series are related to nonresonant $_nF_{n-1}$-hypergeometric series (see Bober\cite[Section~4]{Bo} for the details) so one can apply Beukers-Heckman.

Our approach was motivated by the following observation.  
Let $\Delta(\alpha,\beta)\subseteq {\mathbb R}^m$ be the convex hull of the following set of $m+2$ lattice points: the origin, the standard unit basis vectors, and $(\alpha_1,\dots,\alpha_n,-\beta_{n+1},\dots,-\beta_m)$.  
For a nonnegative integer $k$ and a polytope $\Delta\subseteq{\mathbb R}^m$, we denote by $k\Delta$ the set
\[ k\Delta = \{(kx_1,\dots,kx_m)\in{\mathbb R}^m\mid (x_1,\dots,x_m)\in\Delta\}. \]
It follows easily from a result in \cite{AS5} (see Section 8 for the proof) that the following characterization holds.
\begin{proposition}
The ratios $u_k(\alpha,\beta)$ are integral for all natural numbers~$k$ if and only if the polytope $n\Delta(\alpha,\beta)$ contains no interior lattice points.
\end{proposition}

Thus the series $u(\alpha,\beta;t)$ is an algebraic function if and only if $m=2n+1$ and $n\Delta(\alpha,\beta)$ contains no interior lattice points.  In this paper we apply ideas from toric geometry to explain more directly how the absence of interior lattice points leads to the algebraicity of hypergeometric series.

Consider a more general situation.
Let $B=\{ {\bf b}_1,\dots,{\bf b}_N\}\subseteq {\mathbb Z}^m$ be a finite set of lattice points whose convex hull $\Delta(B)$ is an $m$-dimensional polytope and let
\[ f = \sum_{i=1}^N \lambda_ix^{{\bf b}_i} \in {\mathbb C}(\lambda_1,\dots,\lambda_N)[x_1^{\pm 1},\dots,x_m^{\pm 1}], \]
where $\lambda_1,\dots,\lambda_N$ are indeterminates.
Let ${\mathbb T}^m$ be the $m$-torus over ${\mathbb C}(\lambda)$ and let $X\subseteq{\mathbb T}^m$ be the hypersurface defined by the equation $f = 0$.  Then $X$ is a smooth ${\mathbb C}(\lambda)$-scheme, and we denote by $H^\bullet_{\rm DR}(X)$ its relative de Rham cohomology over~${\mathbb C}(\lambda)$. Denote by ${\mathcal D}$ the ring of differential operators in the $\lambda_i$ with coefficients in~${\mathbb C}(\lambda)$.  Via the Gauss-Manin connection, the cohomology groups $H_{\rm DR}^\bullet(X)$ are modules over the ring~${\mathcal D}$ (by the phrase ``${\mathcal D}$-module'' we always mean a left ${\mathcal D}$-module).  Let $W_\bullet$ denote the weight filtration on de Rham cohomology.  The $W_jH_{\rm DR}^k(X)$ are ${\mathcal D}$-submodules of~$H_{\rm DR}^k(X)$.  

Our focus will be on the top cohomology group $H^{m-1}_{\rm DR}(X)$.  When $m$ is odd, say, $m=2n+1$, we are particularly interested in the case where the Hodge filtration $F^\bullet$ is trivial on $W_{2n}H_{\rm DR}^{2n}(X)$, i.~e., 
\begin{equation}
F^k W_{2n}H_{\rm DR}^{2n}(X) = \begin{cases} W_{2n}H_{\rm DR}^{2n}(X) & \text{if $k\leq n$,} \\ 0 & \text{if $k>n$.} \end{cases}
\end{equation}
The only nonvanishing Hodge number is then the middle one:
\begin{equation}
h^{n,n}\big(W_{2n}H_{\rm DR}^{2n}(X)\big) = \dim_{{\mathbb C}(\lambda)} W_{2n}H_{\rm DR}^{2n}(X).
\end{equation}

The following result is due to N. M. Katz (Theorem 9.1 and Corollary 9.3).
\begin{theorem}
Suppose that $m=2n+1$.  
If the Hodge filtration is trivial on $W_{2n}H_{\rm DR}^{2n}(X)$, then the ${\mathcal D}$-module $W_{2n}H_{\rm DR}^{2n}(X)$ has a full set of solutions that are algebraic functions.
\end{theorem}

Deligne's theory of mixed Hodge structures\cite{De-HodgeII} has been explicitly described for hypersurfaces in a torus by Batyrev\cite{B}.  Using Batyrev's results, we show in Section~2 that if $n\Delta(B)$ contains no interior lattice points, then the Hodge filtration on $W_{2n}H_{\rm DR}^{2n}(X)$ is trivial (Proposition~2.11 below).  Theorem~1.6 then implies our first main result:
\begin{theorem}
Suppose that $m=2n+1$.  
If $n\Delta(B)$ contains no interior lattice points, then the ${\mathcal D}$-module $W_{2n}H_{\rm DR}^{2n}(X)$ has a full set of solutions that are algebraic functions.
\end{theorem}

Theorem 1.7 is a straightforward consequence of the results of Katz and Batyrev.  
But in order to apply Theorem 1.7 to prove that specific series represent algebraic functions, we need to study the ${\mathcal D}$-module $W_{2n}H_{\rm DR}^{2n}(X)$ and its solutions.  The main body of our paper is thus devoted to studying the ${\mathcal D}$-module structure on $H^{m-1}_{\rm DR}(X)$ and to describing solutions of the associated Picard-Fuchs equation.  

In Section 2 we recall Batyrev's description of $H^{m-1}_{\rm DR}(X)$ and its ${\mathcal D}$-module structure.  In Sections 3 and~4, we give an ``$A$-hypergeometric'' description of the Picard-Fuchs equation.  Sections 5 and 7 describe a method (adapted from \cite{AS3}) for finding solutions of the Picard-Fuchs equation.  Sections 6 and 8 contain examples of series that, by Theorem~1.7, are algebraic over ${\mathbb C}(\lambda)$.  Section~9 is the appendix, due to Nicholas Katz.

\section{De Rham cohomology of $X$}

Before recalling his results, we note that Batyrev worked over ${\mathbb C}$ while we are working over ${\mathbb C}(\lambda)$.  Batyrev's results depended on the assumption of nondegeneracy of the polynomial defining the hypersurface.  Since the condition of nondegeneracy is generically satisfied, those results remain valid over ${\mathbb C}(\lambda)$.  For notational convenience we put ${\bf L} = {\mathbb C}(\lambda)[x_1^{\pm 1},\dots,x_m^{\pm 1}]$, the coordinate ring of ${\mathbb T}^m$.  Let $g = x_{m+1}f\in {\bf L}[x_{m+1}]$.

For ${\bf b}_j\in B$, set ${\bf a}_j = ({\bf b}_j,1)\in{\mathbb Z}^{m+1}$ and put $A = \{{\bf a}_1,\dots,{\bf a}_N\}\subseteq{\mathbb Z}^{m+1}$.  Let $\Delta_0(A)\subseteq{\mathbb R}^{m+1}$ be the convex hull of $A\cup\{{\bf 0}\}$, and let $C(A)\subseteq{\mathbb R}^{m+1}$ be the real cone generated by $A$.  Let $M={\mathbb Z}^{m+1}\cap C(A)$, the lattice points contained in $C(A)$, and let $M^\circ\subseteq M$ be the interior lattice points of $C(A)$.

Let $S\subseteq{\bf L}[x_{m+1}]$ be the ${\mathbb C}(\lambda)$-algebra generated by monomials $x^u$ with $u\in M$, let $S^+$ be the ideal of $S$ generated by all $x^u$ with $u\in M\setminus{\bf 0}$, and let ${I}$ be the ideal of $S$ generated by monomials $x^u$ with $u\in M^\circ$.  For $i=1,\dots,m+1$, define differential operators $D_i$ on $S$ by the formula
\begin{equation}
D_i = x_i\frac{\partial}{\partial x_i} + x_i\frac{\partial g}{\partial x_i}.
\end{equation}
By \cite[Theorem 7.13]{B}, we have the isomorphism
\begin{equation}
H^m_{\rm DR}({\mathbb T}^m\setminus X)\cong S^+/\sum_{i=1}^{m+1} D_iS^+
\end{equation}
(note that the variable we are denoting by ``$x_{m+1}$'' was denoted by ``$x_0$'' in \cite{B}).

By \cite[Corollary 5.5 and Theorem 7.5]{B}, we see that under the isomorphism (2.2) the image of $H^m_{\rm DR}({\mathbb T}^m)$ in $H^m_{\rm DR}({\mathbb T}^m\setminus X)$ is spanned by $\sum_{i=1}^{m+1} D_i{\mathbb C}(\lambda)$, hence for the primitive part of cohomology we have
\begin{equation}
PH^m_{\rm DR}({\mathbb T}^m\setminus X)\cong S^+/\sum_{i=1}^{m+1} D_iS.
\end{equation}
The Poincar\'e residue map defines an isomorphism between the primitive cohomologies of $X$ and ${\mathbb T}^m\setminus X$ (\cite[Proposition 5.3]{B}), so we also have
\begin{equation}
PH^{m-1}_{\rm DR}(X)\cong S^+/\sum_{i=1}^{m+1} D_iS.
\end{equation}

By \cite[Theorem~8.2]{B} we have
\begin{equation}
W_{m+1}H^m_{\rm DR}({\mathbb T}^m\setminus X)\cong \bigg({I}+\sum_{i=1}^{m+1} D_iS^+\bigg)/ \sum_{i=1}^{m+1} D_iS^+,
\end{equation}
so by (2.3), (2.4), and the second equation of \cite[Proposition 5.3]{B}
\begin{equation}
W_{m-1}PH^{m-1}_{\rm DR}(X)\cong \bigg({I}+\sum_{i=1}^{m+1} D_iS\bigg)/ \sum_{i=1}^{m+1} D_iS.
\end{equation}
The second sentence of \cite[Corollary 3.10]{B} implies that 
\[ W_{m-1}H^{m-1}_{\rm DR}(X) = W_{m-1}PH^{m-1}_{\rm DR}(X), \]
so we finally have
\begin{equation}
W_{m-1}H^{m-1}_{\rm DR}(X)\cong \bigg(I+\sum_{i=1}^{m+1} D_iS\bigg)/ \sum_{i=1}^{m+1} D_iS.
\end{equation}

The derivations $\partial_j$ corresponding to the variables~$\lambda_j$ act on $H^{m-1}_{\rm DR}(X)$ via the Gauss-Manin connection, making this cohomology group into a ${\mathcal D}$-module.  Let 
\[ D_{\lambda_j} = \frac{\partial}{\partial \lambda_j} + x^{{\bf a}_j}. \]
We make $S$ into a ${\mathcal D}$-module by letting $\partial_j$ act by $D_{\lambda_j}$.  Note that the $D_{\lambda_j}$ are stable on $I$ and commute with the $D_i$, so the right-hand side of (2.7) becomes a ${\mathcal D}$-module.
\begin{proposition}
The isomorphisms $(2.4)$ and $(2.7)$ are isomorphisms of ${\mathcal D}$-modules.
\end{proposition}

\begin{proof}
This result follows from \cite[Theorem 11.6]{B} or \cite[Theorem 1.1]{AS2}.
\end{proof}

We define a grading on the ring $S$ by setting for $u=(u_1,\dots,u_{m+1})\in M$ 
\[ \deg x^u = u_{m+1}. \]
Then $g$ is homogeneous of degree 1 and the quotient ring
\[ \bar{S}:= S/(x_1\partial g/\partial x_1,\dots,x_{m+1}\partial g/\partial x_{m+1})  \]
inherits a grading from $S$.  The graded ring $\bar{S}$ is related to the Hodge filtration $F^\bullet$ on $H^{m-1}_{\rm DR}(X)$ by \cite[Corollary~6.10]{B}:
\begin{equation}
 F^k/F^{k+1} PH^{m-1}_{\rm DR}(X) \cong \bar{S}^{(m-k)}\quad\text{for $k=0,1,\dots,m-1$,} 
\end{equation}
where the isomorphism is induced from (2.4).  Define $\bar{H}$ to be the image of ${I}$ in $\bar{S}$ induced by the inclusion $I\hookrightarrow S$.  Then Batyrev shows\cite[Proposition 9.2]{B} that $\bar{H}$ describes the Hodge filtration on $W_{m-1}H^{m-1}_{\rm DR}(X)$:
\begin{equation}
F^k/F^{k+1}\big(W_{m-1}H^{m-1}_{\rm DR}(X)\big)\cong \bar{H}^{(m-k)} \quad\text{for $k=0,1,\dots,m-1$}.
\end{equation}

\begin{proposition}
Suppose that $m=2n+1$.  If $n\Delta(B)$ contains no interior lattice points, then the Hodge filtration on $W_{2n}H^{2n}_{\rm DR}(X)$ is trivial.
\end{proposition}

\begin{proof}
The hypothesis that $n\Delta(B)$ contains no interior lattice points implies that $\bar{H}^{(i)} = 0$ for $i=1,\dots,n$, hence by (2.10) 
\[ F^{2n+1-k}/F^{2n+2-k} W_{2n}H^{2n}_{\rm DR}(X) = 0\quad\text{for $k=1,\dots,n$.} \]  
Since $F^{2n+1}W_{2n}H^{2n}_{\rm DR}(X) = 0$, we must have
\begin{equation}
F^kW_{2n}H^{2n}_{\rm DR}(X) = 0\quad\text{for $k=n+1,n+2,\dots,2n$.}
\end{equation}
But $W_{2n-1}H^{2n}_{\rm DR}(X)=0$, so Hodge symmetry holds for $W_{2n}H^{2n}_{\rm DR}(X)$.  Since 
\[ F^0W_{2n}H^{2n}_{\rm DR}(X) = W_{2n}H^{2n}_{\rm DR}(X), \]
we conclude from Hodge symmetry and (2.12) that
\begin{equation}
F^kW_{2n}H^{2n}_{\rm DR}(X) = W_{2n}H^{2n}_{\rm DR}(X)\quad\text{for $k=0,1,\dots,n$.}
\end{equation}
Thus the Hodge filtration is trivial on $W_{2n}H^{2n}_{\rm DR}(X)$.
\end{proof}

As noted in Section 1, Proposition 2.11 and Theorem 1.6 immediately imply Theorem~1.7.

{\bf Remark:} If $(n+1)\Delta(B)$ contains no interior lattice points, then $\bar{H}^{(n+1)}=0$ also, so (2.10) implies that (2.12) holds in addition for $k=n$. From (2.13), we conclude that $W_{2n}H^{2n}_{\rm DR}(X) = 0$ in this case.  This leads to the following observation.
\begin{lemma}
Suppose that $m=2n+1$.  
If $W_{2n}H^{2n}_{\rm DR}(X) \neq 0$, then $(n+1)\Delta(B)$ contains an interior lattice point.
\end{lemma}

\section{Picard-Fuchs equations}

In this section, we apply Proposition 2.8 to obtain information about the Picard-Fuchs equation of $W_{m-1}H^{m-1}_{\rm DR}(X)$ and its solutions.  We first recall some general facts.

Let ${\mathcal M}$ be a ${\mathcal D}$-module which is finite-dimensional as a vector space over ${\mathbb C}(\lambda)$.  The classical ``Picard-Fuchs equation'' of ${\mathcal M}$ can be described as the following Pfaffian system.  First choose a basis $\{m_i\}_{i=1}^d$ for ${\mathcal M}$ as ${\mathbb C}(\lambda)$-vector space.  For $k=1,\dots,N$, each $\partial_k(m_i)$ can be written as a linear combination of the $\{m_j\}_{j=1}^d$ with coefficients in ${\mathbb C}(\lambda)$: 
\[ \partial_k(m_i) = \sum_{j=1}^d C^{(k)}_{ij}(\lambda)m_j. \]
Let $C^{(k)}(\lambda) = \big(C^{(k)}_{ij}(\lambda)\big)_{i,j=1}^d$ be the corresponding $(d\times d)$-matrix of rational functions and let $Y$ be the column vector with entries $y_1,\dots,y_d$.  The associated Pfaffian system of differential equations is then
\begin{equation}
\partial_k Y = C^{(k)}(\lambda)Y \quad\text{for $k=1,\dots,N$.}
\end{equation}

The solutions of (3.1) in an arbitrary ${\mathcal D}$-module ${\mathcal F}$ are related to ${\rm Hom}_{\mathcal D}({\mathcal M},{\mathcal F})$ in the following way.  Let $\phi\in{\rm Hom}_{{\mathbb C}(\lambda)}({\mathcal M},{\mathcal F})$ be a vector-space homomorphism.  Then $\phi\in{\rm Hom}_{\mathcal D}({\mathcal M},{\mathcal F})$ if and only if the column vector with entries $\phi(m_1),\dots,\phi(m_d)$ is a solution of the system (3.1).
To say that ${\mathcal M}$ has a full set of algebraic solutions means that all $\phi(m_i)$ are algebraic over ${\mathbb C}(\lambda)$.

Proposition 2.8 shows that $PH^{m-1}_{\rm DR}(X)$ and $W_{m-1}H^{m-1}_{\rm DR}(X)$ are ${\mathcal D}$-submodules of the ${\mathcal D}$-module ${\mathcal W}:=S/\sum_{i=1}^{m+1} D_iS$.  If ${\mathcal F}$ is an arbitrary ${\mathcal D}$-module, we get by restriction maps
\[ {\rm res}:{\rm Hom}_{\mathcal D}({\mathcal W},{\mathcal F})\to \begin{cases} {\rm Hom}_{\mathcal D}(PH^{m-1}_{\rm DR}(X),{\mathcal F}) & \\ {\rm Hom}_{\mathcal D}(W_{m-1}H^{m-1}_{\rm DR}(X),{\mathcal F}). \end{cases} \]
We first describe the elements of ${\rm Hom}_{\mathcal D}({\mathcal W},{\mathcal F})$, which by restriction will then give solutions of the ${\mathcal D}$-modules $PH^{m-1}_{\rm DR}(X)$ and $W_{m-1}H^{m-1}_{\rm DR}(X)$. 

As a  ${\mathbb C}(\lambda)$-vector space, $S$ has basis~$\{x^u\}_{u\in M}$.
Define a ${\mathbb C}(\lambda)$-vector space
\[ R'({\mathcal F}) := \bigg\{ \sum_{u\in M} A_ux^{-u}\mid A_u\in{\mathcal F}\bigg\}. \]
There is a pairing $R'({\mathcal F})\times S\to{\mathcal F}$ defined by
\[ \bigg\langle \sum_{u\in M} A_ux^{-u},\sum_{u\in M} B_ux^u\bigg\rangle = \sum_{u\in M} B_uA_u, \]
where the $B_u$ lie in ${\mathbb C}(\lambda)$ and the sum on the right-hand side is finite because the second sum on the left-hand side is.  This pairing defines an isomorphism
\begin{equation}
{\rm Hom}_{{\mathbb C}(\lambda)}(S,{\mathcal F})\cong R'({\mathcal F}),
\end{equation}
explicitly,
\[ {\rm Hom}_{{\mathbb C}(\lambda)}(S,{\mathcal F})\ni\phi\leftrightarrow \sum_{u\in M} \phi(x^u)x^{-u}\in R'({\mathcal F}). \]

The next step is to determine which elements of $R'({\mathcal F})$ correspond to elements of ${\rm Hom}_{\mathcal D}(S,{\mathcal F})$ under the identification (3.2).  The condition to be satisfied is that for $j=1,\dots,N$ and all $u\in M$
\begin{align*}
 \bigg\langle \sum_{v\in M} A_vx^{-v},D_{\lambda_j}(x^u)\bigg\rangle &= \partial_j\bigg\langle \sum_{v\in M}A_vx^{-v},x^u\bigg\rangle \\
&= \partial_j(A_u).
\end{align*}
But $D_{\lambda_j}(x^u) = x^{u+{\bf a}_j}$, so the left-hand side is just $A_{u+{\bf a}_j}$.  Put
\[ R^*({\mathcal F}) = \bigg\{ \sum_{u\in M} A_ux^{-u}\mid \text{$A_u\in{\mathcal F}$ and $\partial_j(A_u) = A_{u+{\bf a}_j}$ for all $u,j$}\bigg\}. \]
Then we have shown that
\begin{equation}
{\rm Hom}_{\mathcal D}(S,{\mathcal F})\cong R^*({\mathcal F}).
\end{equation}

Equation (3.3) implies that the set ${\rm Hom}_{\mathcal D}({\mathcal W},{\mathcal F})$ can be identified with the elements of $R^*({\mathcal F})$ that annihilate $\sum_{i=1}^{m+1} D_iS$ under the pairing.  Let
\[ \xi = \sum_{u\in M}A_ux^{-u}\in R^*({\mathcal F}). \]
Then $\xi$ vanishes on $\sum_{i=1}^{m+1} D_iS$ if and only if it vanishes on $D_i(x^v)$ for all $i=1,\dots,m+1$ and all $v\in M$.  So the condition to be satisfied is (where we write ${\bf a}_j = (a_{1j},a_{2j},\dots,a_{m+1,j})$)
\begin{align*}
\langle\xi,D_i(x^v)\rangle &= \bigg\langle \xi,v_ix^v + \sum_{j=1}^N a_{ij}\lambda_jx^{v+{\bf a}_j}\bigg\rangle \\
 & = v_iA_v + \sum_{j=1}^N a_{ij}\lambda_jA_{v+{\bf a}_j} \\
 & = 0.
\end{align*}
But since $\xi\in R^*({\mathcal F})$ we have $A_{v+{\bf a}_j} = \partial_j A_v$, so this condition becomes
\[ \bigg(\sum_{j=1}^N a_{ij}\lambda_j\partial_j\bigg)A_v = -v_iA_v \quad\text{for all $v\in M$ and $i=1,\dots,m+1$.} \]
To simplify notation, we write this condition in vector format:
\begin{equation}
\bigg(\sum_{j=1}^N {\bf a}_j\lambda_j\partial_j\bigg)A_v = -vA_v \quad\text{for all $v\in M$.}
\end{equation}
Let
\begin{multline*}
 {\mathcal K}({\mathcal F}) = \bigg\{ \sum_{u\in M}A_ux^{-u}\mid A_u\in{\mathcal F},\; \partial_i(A_u) = A_{u+{\bf a}_i} \,\text{for}\, i=1,\dots,N,\\ \text{and}\; \bigg(\sum_{j=1}^N {\bf a}_j\lambda_j\partial_j\bigg)(A_u) = -uA_u\;\text{for all $u\in M$}\bigg\}. 
\end{multline*}
Then we have proved that there is an isomorphism
\begin{equation}
{\rm Hom}_{\mathcal D}({\mathcal W},{\mathcal F})\cong {\mathcal K}({\mathcal F}).
\end{equation}

\begin{proposition}
Suppose that $\sum_{u\in M} A_ux^{-u}\in{\mathcal K}({\mathcal F})$.  Then 
\[ \sum_{u\in M^\circ} A_ux^{-u}\in{\rm Hom}_{\mathcal D}(W_{m-1}H^{m-1}_{\rm DR}(X),{\mathcal F}). \]
\end{proposition}

\begin{proof}
This follows immediately from (3.5) and Proposition 2.8: $\sum_{u\in M^\circ} A_ux^{-u}$ is just the restriction of $\sum_{u\in M} A_ux^{-u}\in{\rm Hom}_{\mathcal D}({\mathcal W},{\mathcal F})$ to $W_{m-1}H^{m-1}_{\rm DR}(X)$.
\end{proof}

\begin{corollary}
Suppose that $m=2n+1$.  If $\sum_{u\in M} A_ux^{-u}\in{\mathcal K}({\mathcal F})$ and the polytope $n\Delta(B)$ contains no interior lattice points, then the $A_u$ for $u\in M^\circ$ are algebraic over ${\mathbb C}(\lambda)$.
\end{corollary}

\begin{proof}
Fix $v\in M^\circ$.  If $A_v\neq 0$, then $x^v\not\in I\cap\sum_{i=1}^{m+1} D_iS$, so one can find a basis for $W_{m-1}H^{m-1}_{\rm DR}(X)$ (under the identification (2.7)) that contains $x^v$.  By the general remarks at the beginning of this section, $A_v$ will appear in the solution vector of the Picard-Fuchs equation of $W_{m-1}H^{m-1}_{\rm DR}(X)$ corresponding to $\sum_{u\in M^\circ} A_ux^{-u}$.  The assertion of the corollary now follows from Theorem~1.7.
\end{proof}

{\bf Remark.}  Spaces of the type ${\mathcal K}({\mathcal F})$ were first introduced by Dwork\cite{D}.  Roughly speaking, he considered (a $p$-adic analogue of) the space ${\mathcal K} := {\rm Hom}_{{\mathbb C}(\lambda)}({\mathcal W},{\mathbb C}(\lambda))$.  One then has 
\[ {\mathcal K}\bigotimes_{{\mathbb C}(\lambda)} {\mathcal F} \cong {\rm Hom}_{{\mathbb C}(\lambda)}({\mathcal W},{\mathcal F}). \]
Dwork described ${\rm Hom}_{{\mathcal D}}({\mathcal W},{\mathcal F})$ as a subset of ${\mathcal K}\bigotimes_{{\mathbb C}(\lambda)} {\mathcal F}$.  

\section{The $A$-hypergeometric nature of ${\mathcal K}({\mathcal F})$}

We recall the definition of the $A$-hypergeometric system with parameter $\beta\in{\mathbb C}^{m+1}$.
Let $L$ be the lattice of relations on the set $A$:
\[ L = \bigg\{ l=(l_1,\dots,l_N)\in{\mathbb Z}^N\mid \sum_{j=1}^N l_j{\bf a}_j = {\bf 0}\bigg\}. \]
For $l\in L$, define the box operator $\Box_l$ by
\begin{equation} \Box_l = \prod_{l_j>0} \bigg(\frac{\partial}{\partial\lambda_j}\bigg)^{l_j} -  \prod_{l_j<0} \bigg(\frac{\partial}{\partial\lambda_j}\bigg)^{-l_j}. 
\end{equation}

Let $\beta = (\beta_1,\dots,\beta_{m+1})\in{\mathbb C}^{m+1}$.  Associated to $\beta$ and the set $A$ are the Euler or homogeneity operators, written in vector format as
\begin{equation}
 \sum_{j=1}^N {\bf a}_j\lambda_j\partial_j - \beta. 
\end{equation}
The {\it $A$-hypergeometric system with parameter $\beta$\/} is the system consisting of the box operators (4.1) for $l\in L$ and the Euler operators (4.2).  

\begin{proposition}
Let $\sum_{u\in M} A_ux^{-u}\in{\mathcal K}({\mathcal F})$.  Then $A_u$ satisfies the $A$-hyper\-geomet\-ric system with parameter $\beta = -u$.
\end{proposition}

\begin{proof}
It follows immediately from the definition of ${\mathcal K}({\mathcal F})$ that each $A_u$ satisfies the Euler operators with parameter $\beta = -u$.

The condition that $\partial_j(A_u) = A_{u+{\bf a}_j}$ implies that
\begin{equation}
\Box_l(A_u) = A_{u+ \sum_{l_j>0} l_j{\bf a}_j} - A_{u-\sum_{l_j<0} l_j{\bf a}_j}. 
\end{equation}
But $l\in L$ implies that $\sum_{l_j>0} l_j{\bf a}_j = -\sum_{l_j<0} l_j{\bf a}_j$, so the right-hand side of (4.4) vanishes.  It follows that if $\sum_{u\in M} A_ux^{-u}\in R^*({\mathcal F})$, then $A_u$ satisfies $\Box_l$ for all $u\in M$ and $l\in L$.
\end{proof}

\section{Explicit solutions}

Our scheme for obtaining algebraic solutions of $W_{2n}H^{2n}_{\rm DR}(X)$ will be to construct elements of ${\mathcal K}({\mathcal F})$ and apply Corollary~3.7.  We start by recalling the procedure of~\cite{AS3}, which constructs elements of ${\mathcal K}({\mathcal F})$ for certain ${\mathcal D}$-modules ${\mathcal F}$ whose elements are formal Laurent series multiplied by complex powers of the variables.

For $z\in{\mathbb C}$ and $k\in{\mathbb Z}$, $k<-z$ if $z\in{\mathbb Z}_{<0}$, define
\[ [z]_k = \begin{cases} 1 & \text{if $k=0$,} \\ \displaystyle \frac{1}{(z+1)(z+2)\cdots(z+k)} & \text{if $k>0$,} \\
z(z-1)\cdots(z+k+1) & \text{if $k<0$.} \end{cases} \]
For $z=(z_1,\dots,z_N)\in{\mathbb C}^N$ and $k=(k_1,\dots,k_N)\in{\mathbb Z}^N$, we define, assuming for every $i$ that $k_i<-z_i$ if $z_i\in{\mathbb Z}_{<0}$,
\[ [z]_k = \prod_{i=1}^N [z_i]_{k_i}. \]

Let $\beta\in{\mathbb Z}^{m+1}$ and fix $v=(v_1,\dots,v_N)\in\big({\mathbb C}\setminus{\mathbb Z}_{<0}\big)^N$ such that
\begin{equation}
\sum_{i=1}^N v_i{\bf a}_i = \beta. 
\end{equation}
Let $T_1,\dots,T_{m+1}$ be indeterminates.  If ${\bf c} = (c_1,\dots,c_{m+1})\in{\mathbb C}^{m+1}$, we write $T^c = T_1^{c_1}\cdots T_{m+1}^{c_{m+1}}$.  Form the generating series
\[ \Phi_{v_i}(\lambda_i,T) = \sum_{k_i\in{\mathbb Z}} [v_i]_{k_i}\lambda_i^{v_i+k_i}T^{(v_i+k_i){\bf a}_i} \]
and their product 
\[ \Phi_v(\lambda,T) = \prod_{i=1}^N \Phi_{v_i}(\lambda_i,T)  = \sum_{u\in{\mathbb Z}^{m+1}} \Phi_{v,u}(\lambda)T^{\beta+u}, \]
where
\begin{equation}
\Phi_{v,u}(\lambda) = \sum_{\sum k_i{\bf a}_i = u} [v]_k\lambda^{v+k}
\end{equation}
and we write $k=(k_1,\dots,k_N)$ and $\lambda^{v+k} = \lambda_1^{v_1+k_1}\cdots\lambda_N^{v_N+k_N}$.  

By \cite[Equation (2.12)]{AS3} (or a straightforward calculation from the definition) we have 
\[ \partial_j \Phi_v(\lambda,T) = T^{{\bf a}_j}\Phi_v(\lambda,T), \]
i.~e., 
\begin{equation}
\partial_j\Phi_{v,u}(\lambda) = \Phi_{v,u-{\bf a}_j}(\lambda).
\end{equation}
When $\sum_{i=1}^N k_i{\bf a}_i = u$, we have
\[ \sum_{i=1}^N (v_i+k_i){\bf a}_i = \beta+u. \]
The Euler operator $\sum_{i=1}^N {\bf a}_i\lambda_i\partial_i - (\beta+u)$ thus annihilates every monomial $\lambda^{v+k}$ on the right-hand side of (5.2), so it annihilates $\Phi_{v,u}(\lambda)$.   

For $u\in M$, set 
\begin{equation}
A_u = \Phi_{v,-\beta-u}(\lambda) = \sum_{\sum k_i{\bf a}_i = -\beta-u} [v]_k\lambda^{v+k}.  
\end{equation}
Then $A_u$ satisfies the Euler operator $\sum_{i=1}^N {\bf a}_i\lambda_i\partial_i + u$, and by (5.3) we have $\partial_j A_u = A_{u+{\bf a}_j}$.  This implies that (4.4) holds, so the proof of Proposition 4.3 shows that the $A_u$ satisfy the box operators $\Box_l$ for $l\in L$.  We have proved the following proposition.

\begin{proposition}
Let $A_u$ be given by (5.4) for $u\in M$.    If all $A_u$ lie in some ${\mathcal D}$-module ${\mathcal F}$, then the series $\sum_{u\in M} A_ux^{-u}$ lies in ${\mathcal K}({\mathcal F})$.  
\end{proposition}

Applying Corollary 3.7 then gives the following result.
\begin{corollary}
Suppose that $m=2n+1$ and that the $A_u$ given by~(5.4) all lie in some ${\mathcal D}$-module ${\mathcal F}$.   
If $n\Delta(B)$ contains no interior lattice points, then $A_u$ is algebraic over ${\mathbb C}(\lambda)$ for $u\in M^\circ$.
\end{corollary}

Whether there exists a ${\mathcal D}$-module ${\mathcal F}$ containing all the series $A_u$ depends on the choice of $v$ satisfying (5.1).  For example, if we take $\beta = {\bf 0}$ and all $v_i=0$, then $A_{\bf 0} = 1$ and $A_u=0$ for $u\in M$, $u\neq {\bf 0}$.  In this case, all $A_u$ lie in the ${\mathcal D}$-module ${\mathbb C}(\lambda)$.  

In general we shall take ${\mathcal F}$ to be a ${\mathcal D}$-module of the following type.  Let $C$ be a strongly convex rational polyhedral cone in ${\mathbb R}^N$, i.~e., the cone $C$ is generated by a finite set of integral vectors and has a vertex at the origin.  Let $H_C$ be the set of all series
\[ \xi = \sum d_v\lambda^v\in {\mathbb C}[[\lambda_1^{\pm 1},\dots,\lambda_N^{\pm 1}]] \]
with the following property:  there exists ${\bf c}_\xi\in{\mathbb Z}^N$ such that if $d_v\neq 0$, then
\[ v\in {\mathbb Z}^N\cap({\bf c}_\xi+C). \]
Multiplication of two elements of $H_C$ is clearly well defined, and it is straightforward to check that $H_C$ is an integral domain.   The derivations $\partial_i$ act on $H_C$ in the obvious way.  Furthermore, $H_C$ contains ${\mathbb C}[\lambda]$, so its quotient field $H'_C$ contains ${\mathbb C}(\lambda)$ and hence becomes a ${\mathcal D}$-module.  In general, we shall look for solutions of ${\mathcal W}$ in a ${\mathcal D}$-module ${\mathcal F}$ of the form
\[ {\mathcal F} = H'_C[\lambda^{{\bf q}},\log\lambda_1,\dots,\log\lambda_N], \]
where ${\bf q}$ lies in ${\mathbb Q}^N$.  

\section{Examples}

In this section we apply the results of Section 5 to construct a class of $A$-hypergeometric series that are algebraic functions.  
Recall that the {\it negative support\/} of a vector $v=(v_1,\dots,v_N)\in{\mathbb C}^N$ is the set
\[ {\rm nsupp}(v) = \{i\in\{1,\dots,N\}\mid v_i\in{\mathbb Z}_{<0}\}. \]
The vector $v$ is said to have {\it minimal negative support\/} if ${\rm nsupp}(v+l)$ is not a proper subset of ${\rm nsupp}(v)$ for any $l\in L$, where $L$ is the lattice of relations on the set $A$ (see Section 4).  Let 
\[ L_v = \{l\in L\mid {\rm nsupp}(v+l) = {\rm nsupp}(v)\}. \]
If $v$ has minimal negative support and $\sum_{i=1}^N v_i{\bf a}_i = \beta\in{\mathbb C}^{m+1}$, then by Saito, Sturmfels, and Takayama\cite[Proposition 3.4.13]{SST} the series
\begin{equation}
\Psi_v(\lambda) := \sum_{l\in L_v} [v]_l\lambda^{v+l}\in \lambda^v{\mathbb C}[[\lambda_1^{\pm 1},\dots,\lambda_N^{\pm 1}]] 
\end{equation}
is a formal solution of the $A$-hypergeometric system with parameter $\beta$.  

We first note a general property of interior lattice points.  Let $B$ be as in the Introduction and let $C(A)$ be as in Section~2.  Let $k\in{\mathbb Z}_{>0}$ be minimal with the property that $k\Delta(B)$ contains an interior lattice point~$u$.  Then $u^{(0)}:=(u,k)$ is a lattice point in the interior of $C(A)$ whose $(m+1)$-st coordinate is minimal among all interior lattice points of~$C(A)$.  
\begin{lemma}
Suppose that $v_1,\dots,v_N\in{\mathbb R}_{\geq 0}$ satisfy the equation
\[ \sum_{i=1}^N v_i{\bf a}_i = u^{(0)}. \]
Then $v_i\leq 1$ for all $i$.
\end{lemma}

\begin{proof}
Suppose that $v_i>1$ for some $i$, say, $v_1>1$.  Put
\[ v'_i = \begin{cases} v_i & \text{for $i=2,\dots,N$,} \\ v_1-1 & \text{for $i=1$,} \end{cases} \]
and consider
\[ u^{(0)}-{\bf a}_1 = \sum_{i=1}^N v'_i{\bf a}_i. \]
The $v'_i$ are all nonnegative and $v'_i>0$ if $v_i>0$, so $u^{(0)}-{\bf a}_1$ is an interior lattice point of $C(A)$ whose $(m+1)$-st coordinate equals $k-1$, contradicting the minimality property of $u^{(0)}$.
\end{proof}

We examine more closely the behavior of the series $\Psi_v(\lambda)$ of (6.1) when $m=2n+1$ and $n\Delta(B)$ contains no interior lattice points, which we assume for the remainder of this section.  To avoid the trivial case (see Lemma~2.14) we assume that $(n+1)\Delta(B)$ does contain interior lattice points.  We show how to associate an algebraic function to each interior lattice point of $(n+1)\Delta(B)$ that satisfies the hypothesis of Proposition~6.4 below.

Let $u$ be an interior lattice point of $(n+1)\Delta(B)$.  Then $u^{(0)}=(u,n+1)$ is an interior lattice point of $C(A)$ whose last coordinate $n+1$ is minimal among all interior lattice points.  Since $C(A)\subseteq{\mathbb R}^{2n+2}$, the point $u^{(0)}$ lies in the real cone generated by $2n+2$ linearly independent elements of $A$, say, $\{{\bf a}_i\}_{i=1}^{2n+2}$.  We may thus write
\begin{equation}
 \sum_{i=1}^{2n+2} v_i{\bf a}_i = -u^{(0)} 
\end{equation}
with $v_1,\dots,v_{2n+2}\in{\mathbb Q}_{\leq 0}$.  By Lemma 6.2 all $v_i$ lie in the interval $[-1,0]$.  Put $v=(v_1,\dots,v_{2n+2},0,\dots,0)\in [-1,0]^N$ and consider the associated series $\Psi_v(\lambda)$ defined by~(6.1).   

\begin{proposition}
Suppose that $m=2n+1$ and that $n\Delta(B)$ contains no interior lattice points.
If $v_i>-1$ for $i=1,\dots,2n+2$, then $\Psi_v(\lambda)$ is a solution of the $A$-hypergeometric system with parameter $-u^{(0)}$ and its coefficients are $p$-integral for every prime $p$ for which all the $v_i$ are $p$-integral.
\end{proposition}

\begin{proof}
The hypothesis that $v_i>-1$ for all $i$ implies that ${\rm nsupp}(v) = \emptyset$, so $v$ has minimal negative support.  Equation (6.3) then implies that $\Psi_v(\lambda)$ is a solution of the $A$-hypergeometric system with parameter~$-u^{(0)}$.  

Fix a prime number $p$ for which all $v_i$ are $p$-integral.  For a $p$-integral rational number $r$, $-1<r\leq 0$, define $r'$ to be the unique $p$-integral rational number, $-1<r'\leq 0$, satisfying $pr'-r\in{\mathbb Z}$.  We denote the $k$-fold iteration of this operation by $r^{(k)}$.  
We claim that $-\sum_{i=1}^{2n+2} v'_i{\bf a}_i$ is an interior lattice point of $C(A)$ whose last coordinate equals $n+1$.  

To see this, for $i=1,\dots,2n+2$ let $\epsilon_i\in\{0,-1,\dots,-(p-1)\}$ be chosen so that $v_i+\epsilon_i = pv_i'$.  Then by (6.3)
\[ -\sum_{i=1}^{2n+2} pv'_i{\bf a}_i =u^{(0)} - \sum_{i=1}^{2n+2} \epsilon_i{\bf a}_i. \]
The right-hand side of this equation shows that this vector has integral coordinates and the left-hand side shows that all these coordinates are divisible by $p$, hence $-\sum_{i=1}^{2n+2}v'_i{\bf a}_i$ has integral coordinates.  Furthermore, $-v'_i>0$ whenever $-v_i>0$, so $-\sum_{i=1}^{2n+2}v'_i{\bf a}_i$ is an interior lattice point of~$C(A)$.

Since the last coordinate of each ${\bf a}_i$ equals 1, the last coordinate of $-\sum_{i=1}^{2n+2}v'_i{\bf a}_i$ equals $-\sum_{i=1}^{2n+2} v'_i$.  And since the minimal last coordinate of any interior lattice point of $C(A)$ is $n+1$, we have
\[ -\sum_{i=1}^{2n+2}v'_i\geq n+1. \]
But if this inequality were strict, then $\sum_{i=1}^{2n+2} (1+v'_i)<n+1$, which implies that $\sum_{i=1}^{2n+2} (1+v'_i){\bf a}_i$ is an interior lattice point of $C(A)$ with last coordinate $<n+1$, a contradiction.

Induction on $k$ then shows that for all $k$, $-\sum_{i=1}^{2n+2} v_i^{(k)}{\bf a}_i$ is an interior lattice point of $C(A)$ whose last coordinate equals $n+1$.  The $p$-integrality of the coefficients of $\Psi_v(\lambda)$ now follows by \cite[Corollary~3.6]{AS}.
\end{proof}

By a theorem of Eisenstein, the coefficients of a series representing an algebraic function are $p$-integral for all but a finite number of primes $p$.  We show in Theorem~6.6 below that the series $\Psi_v(\lambda)$ of Proposition~6.4 is in fact algebraic.  The algebraic series described by Rodriguez-Villegas have integral coefficients and do not arise as special cases of Proposition~6.4 (they are discussed in Section 8).  However, they are combinatorially related to many algebraic series that do arise as special cases of Proposition~6.4.  These other series may be constructed using Bober's list\cite{Bo} of all ratios (1.1) with $m=2n+1$ which are integral for all $k\in{\mathbb Z}_{\geq 0}$.  The following example is based on \cite[Table~2, Line~11]{Bo}.

{\bf Example 1:}  Let $B = \{{\bf b}_i\}_{i=1}^7\subseteq{\mathbb Z}^5$, where ${\bf b}_1,\dots,{\bf b}_5$ are the standard unit basis vectors in ${\mathbb R}^5$, ${\bf b}_6 = (9,1,-5,-3,-2)$, 
and ${\bf b}_7$ is the zero vector.  Let $A = \{{\bf a}_i\}_{i=1}^7\subseteq{\mathbb Z}^6$, where ${\bf a}_i = ({\bf b}_i,1)$.  One checks that the polytope $2\Delta(B)$ contains no interior lattice points but that $3\Delta(B)$ does, for example, $u=(2,1,-1,0,0,)$ is an interior point of $3\Delta(B)$.  The augmented vector $u^{(0)} = (u,3)$ is interior to the cone spanned by $\{{\bf a}_i\}_{i=2}^7$ and we have
\[ \sum_{i=1}^7 v_i{\bf a}_i = -u^{(0)} \]
for
\[ v = (0,-7/9, -1/9,-2/3,-4/9,-2/9,-7/9). \]
Note that the lattice $L$ of relations on the set $A$ is given by
\[ L = \{ l(9,1,-5,-3,-2,-1,1)\mid l\in{\mathbb Z}\}, \]
and for our choice of $v$ we have
\[ L_v = \{ l(9,1,-5,-3,-2,-1,1)\mid l\in{\mathbb Z}_{\geq 0}\}. \]
In this case, the series $\Psi_v(\lambda)$ of Proposition 6.4 becomes
\[ \sum_{l=0}^\infty \textstyle [0]_{9l}[-\frac{7}{9}]_l[-\frac{1}{9}]_{-5l}[-\frac{2}{3}]_{-3l}[-\frac{4}{9}]_{-2l}][-\frac{2}{9}]_{-l}
[-\frac{7}{9}]_l \lambda^{v+l(9,1,-5,-3,-2,-1,1)}. \]
When we rewrite this series in terms of the classical Pochhammer symbol \big($(a)_k = a(a+1)\cdots(a+k-1)$\big) it becomes
\begin{equation}
 \sum_{l=0}^\infty (-1)^l \frac{ (\frac{1}{9})_{5l} (\frac{2}{3})_{3l} (\frac{4}{9})_{2l}} {(1)_{9l} (\frac{2}{9})_l} \lambda^{v+l(9,1,-5,-3,-2,-1,1)}. 
\end{equation}
By Proposition 6.4, this series has $p$-integral coefficients for all primes $p\neq 3$.  

\begin{theorem}
Under the hypothesis of Proposition 6.4, the series $\Psi_v(\lambda)$ is an algebraic function.
\end{theorem}

We prove Theorem 6.6 by constructing a sequence $\{A_u\}_{u\in M^\circ}$ in which $\Psi_v(\lambda)$ appears and applying Corollary 5.6.  We begin with an elementary lemma whose proof is left to the reader.

Let ${\bf d}_1,\dots,{\bf d}_k$ be linearly independent vectors in ${\mathbb R}^{2n+2}$, let ${\bf e}_1,\dots,{\bf e}_M$ be vectors in ${\mathbb R}^{2n+2}$, and suppose that the real cone $E$ generated by $\{{\bf e}_i\}_{i=1}^M$ has a vertex at the origin.  Let ${\bf w}\in{\mathbb R}^{2n+2}$.  Define
\[ E'_{\bf w} = \bigg\{ (y_1,\dots,y_k)\in{\mathbb R}^k\mid \sum_{i=1}^k y_i{\bf d}_i\in {\bf w} + E\bigg\}. \]
\begin{lemma}
The set $E'_{\bf 0}$ is a cone with a vertex at the origin, and $E'_{\bf w} = {\bf w}' + E'_{\bf 0}$ for some ${\bf w}'\in{\mathbb R}^k$.  
\end{lemma}

\begin{proof}[Proof of Theorem 6.6]
Let the $v_i$ satisfy (6.3) with $v_i>-1$ for $i=1,\dots,2n+2$.   Set 
\[ v'_i = \begin{cases} v_i-1 & \text{if $v_i<0$,} \\ 0 & \text{if $v_i=0$.} \end{cases} \]
Set $u^{(1)}=\sum_{v_i<0} {\bf a}_i$ and define
\begin{equation}
 \beta = \sum_{i=1}^{2n+2} v'_i{\bf a}_i = -u^{(0)}-u^{(1)}. 
\end{equation}
Note that $-\beta$ is an interior lattice point of $C(A)$ since $u^{(0)}$ is.  We apply the construction of Section 5 with (5.1) replaced by (6.8).  Equation~(5.2) then becomes
\begin{equation}
\Phi_{v',u}(\lambda) = \sum_{\sum k_i{\bf a}_i = u} [v']_k\lambda^{v'+k}
\end{equation}
for $u\in{\mathbb Z}^{2n+2}$.  Equation (5.4) becomes
\begin{equation}
A_u = \Phi_{v',-\beta-u}(\lambda) = \sum_{\sum k_i{\bf a}_i = -\beta-u} [v']_k\lambda^{v'+k}.
\end{equation}

Note first the relation between $A_{u^{(0)}}$ and the series $\Psi_v(\lambda)$ we assert is an algebraic function.  From (6.10) we have
\begin{equation}
A_{u^{(0)}} = \sum_{\sum k_i{\bf a}_i = u^{(1)}} [v']_k\lambda^{v'+k}.
\end{equation}
Since $-1<v_i\leq 0$ for all $i$, if $l\in L\setminus L_v$, then there exists $i$ such that $v_i=0$ and $l_i<0$.  This implies that $[v_i]_{l_i} = 0$, hence $[v]_l = 0$.  It follows that the sum over $L_v$ in (6.1) can be replaced by a sum over $L$:
\begin{equation}
\Psi_v(\lambda) = \sum_{\sum l_i{\bf a}_i = {\bf 0}} [v]_l\lambda^{v+l}.
\end{equation}
There is a bijective correspondence between the index sets in the summations in (6.11) and (6.12):  one has $\sum_{i=1}^N l_i{\bf a}_i = {\bf 0}$ if and only if 
\[ \sum_{v_i<0} (l_i+1){\bf a}_i + \sum_{v_i=0} l_i{\bf a}_i = u^{(1)}. \]
If $r\in{\mathbb C}\setminus{\mathbb Z}_{<0}$ and $m\in {\mathbb Z}$, then $r[r-1]_{m+1} = [r]_m$.  A short calculation then shows that
\[ \Psi_v(\lambda) = \big(\prod_{v_i<0} v_i\big)A_{u^{(0)}}, \]
so if we can show that the $\{A_u\}_{u\in M}$ all lie in a ${\mathcal D}$-module ${\mathcal F}$, then Corollary~5.6 implies that $\Psi_v(\lambda)$ is an algebraic function.

As noted at the end of Section 5, it suffices to show that for each $u\in M$, the exponents $k=(k_1,\dots,k_N)\in{\mathbb Z}^N$ that appear on the right-hand side of (6.10) lie in a translate (depending on $u$) of a real cone with vertex at the origin.  We have $v'_i=0$ if $v_i=0$ and $[0]_{k_i}=0$ if $k_i<0$, so if $[v']_k\neq 0$, then
\begin{equation}
k_i\geq 0\quad\text{for $v_i=0$.}
\end{equation}
The sum on the right-hand side of (6.10) then runs over solutions of the equation
\begin{equation}
\sum_{v_i<0} k_i{\bf a}_i = u^{(0)} + u^{(1)} -u -\sum_{v_i=0} k_i{\bf a}_i
\end{equation}
subject to (6.13).  

Put $r={\rm card}\{{\bf a}_i\mid v_i<0\}$.  Let $E\subseteq{\mathbb R}^{2n+2}$ be the real cone generated by $\{-{\bf a}_i\}_{v_i=0}$ and set
\[ E'_u = \bigg\{ (k_i)_{v_i<0}\in{\mathbb R}^r\mid \sum_{v_i<0} k_i{\bf a}_i\in u^{(0)}+u^{(1)}-u+E\bigg\}. \]
By Lemma 6.7 we have $E'_u = u'+E''$ for some $u'\in{\mathbb R}^r$, where $E''$ is a cone in ${\mathbb R}^r$ with a vertex at the origin.
It follows that all the solutions $\big((k_i)_{v_i<0},(k_i)_{v_i=0}\big)\in{\mathbb Z}^N$ of (6.14) subject to (6.13) lie in a translate of $E''\times ({\mathbb R}_{\geq 0})^{N-r}$, which is a cone in ${\mathbb R}^N$ with a vertex at the origin.
\end{proof}

The proof of Theorem 6.6 shows more than the algebraicity of $\Psi_v(\lambda)$.
\begin{corollary}
If $v_i>-1$ for $i=1,\dots,2n+2$, then for $u\in M^\circ$ the series $A_u$ defined by $(6.10)$ are algebraic functions.
\end{corollary}

\section{Formal logarithmic solutions}

The series (1.2) that originally stimulated our interest are related to logarithmic solutions of ${\mathcal W}$.  Before examining them in detail, we make some general remarks on logarithmic solutions of $A$-hypergeometric systems that extend the results of \cite{AS3} in a special case.  We begin by defining some generating series whose significance will gradually become apparent.

Fix a nonnegative integer $r$.
There is a unique sequence $\{f^{(r)}_k(t)\}_{k\in{\mathbb Z}}$ of functions of one variable~$t$ satisfying the three conditions (i) $f^{(r)}_0(t) = \log^r t$, (ii) $d/dt\big(f^{(r)}_k(t)\big) = f^{(r)}_{k-1}(t)$ for all $k\in{\mathbb Z}$, and (iii) $f^{(k)}(t)$ equals $t^k$ times a polynomial in $\log t$ of degree $\leq r$ with constant coefficients.  One calculates this sequence by successive differentiation and integration of $\log^r t$.  We summarize the result.
Let $i$ be a nonnegative integer.  For a positive integer $k$, set
\[ M_{k,i} = \frac{(-1)^i}{k!}\sum_{j_1+\cdots+j_k = i} 1^{-j_1}2^{-j_2}\cdots k^{-j_k},\]
for $k=0$, set
\[ M_{0,i} = \begin{cases} 1 & \text{if $i=0$,} \\ 0 & \text{if $i>0$,} \end{cases} \]
and for a negative integer $k$, set
\[ M_{k,i} = (-1)^{-k-i}\sum_{1\leq j_1<\dots<j_{i-1}\leq -k-1} \frac{(-k-1)!}{j_1j_2\cdots j_{i-1}}, \]
with the understanding that (for $k<0$) $M_{k,i} = 0$ for $i=0$ or $i>-k$.  Then
\begin{equation}
f^{(r)}_k(t) = t^k\sum_{i=0}^r M_{k,i}r(r-1)\cdots(r-i+1)\log^{r-i}t.
\end{equation}

For a nonnegative integer $r$, put ${\mathcal P}_r=\{1,\dots,N\}^r$, the set of sequences of length $r$ from the set $\{1,\dots,N\}$.  Let $P = (p_1,\dots,p_r) \in{\mathcal P}$.  Up to ordering, the sequence $P$ is determined by its associated multiplicity function $\rho_P:\{1,\dots,N\}\to{\mathbb N}$ defined by
\[ \rho_P(i) = {\rm card}\{j\mid p_j=i\}. \]
We associate to $P$ a generating series $\Psi^P(\lambda,T)$ (where $T=(T_1,\dots,T_{m+1})$) defined as follows.
First put 
\begin{equation}
\Psi_{\rho_P(i)}(\lambda_i,T) = \sum_{k_i\in{\mathbb Z}} f^{(\rho_P(i))}_{k_i}(\lambda_i)T^{k_i{\bf a}_i},
\end{equation}
then set
\begin{equation}
\Psi^P(\lambda,T) = \prod_{i=1}^N \Psi_{\rho_P(i)}(\lambda_i,T).
\end{equation}
This series depends only on $\rho_P$, not on $P$, i.~e., it is independent of the ordering of the sequence~$P$.  However, it will be useful for bookkeeping purposes to index these series by $P$.  We use these series to construct logarithmic solutions of the $A$-hypergeometric system with parameters in ${\mathbb Z}$.

Write 
\begin{equation}
\Psi^P(\lambda,T) = \sum_{u\in{\mathbb Z}^{m+1}} \Psi^P_u(\lambda)T^u.
\end{equation}
Equation (7.2) and the definition of $f^{(\rho_P(i))}_{k_i}(\lambda_i)$ imply that
\[ \frac{\partial}{\partial\lambda_i} \Psi_{\rho_P(i)}(\lambda_i,T) = T^{{\bf a}_i}\Psi_{\rho_P(i)}(\lambda_i,T), \]
hence by (7.3)
\[ \frac{\partial}{\partial\lambda_i} \Psi^P(\lambda,T) = T^{{\bf a}_i}\Psi^P(\lambda,T), \]
or, equivalently,
\begin{equation}
\frac{\partial}{\partial\lambda_i}\Psi^P_u(\lambda) = \Psi^P_{u-{\bf a}_i}(\lambda).
\end{equation}
Arguing as in the proof of Proposition 4.3 (see Eq.~(4.4)) then gives the following result.
\begin{proposition}
For all $u\in{\mathbb Z}^{m+1}$, the expression $\Psi^P_u(\lambda)$ satisfies the box operators $\Box_l$ for $l\in L$.
\end{proposition}

In general, the $\Psi^P_u(\lambda)$ do not satisfy any Euler operators because of the presence of log terms.  We shall analyze the $\Psi^P_u(\lambda)$ more closely to see how to create logarithmic series that also satisfy Euler operators.

From (7.2), (7.3), and (7.4), we compute
\begin{multline}
\Psi^P_u(\lambda) = \sum_{\substack{k_1,\dots,k_N\in{\mathbb Z}\\ \sum k_i{\bf a}_i = u}} \lambda_1^{k_1}\cdots \lambda_N^{k_N} \cdot \\
\prod_{i=1}^N \bigg(\sum_{j_i=0}^{\rho_P(i)} M_{k_i,j_i} \rho_P(i)(\rho_P(i)-1)\cdots(\rho_P(i)-j_i+1) \log^{\rho_P(i)-j_i}\lambda_i\bigg).
\end{multline}

Let $\rho,\rho':\{1,\dots,N\}\to{\mathbb N}$ be arbitrary functions.  We write $\rho'\leq\rho$ if $\rho'(i)\leq\rho(i)$ for all~$i$.  Put
\[ R(\rho_P) = \{ \rho:\{1,\dots,N\}\to{\mathbb N}\mid \rho\leq\rho_P\}. \]
We may rewrite (7.7) as
\begin{multline}
\Psi^P_u(\lambda) = \sum_{\substack{k_1,\dots,k_N\in{\mathbb Z}\\ \sum k_i{\bf a}_i = u}} \lambda_1^{k_1}\cdots \lambda_N^{k_N} \cdot \\
\sum_{\rho\in R(\rho_P)} \prod_{i=1}^N M_{k_i,\rho(i)} \rho_P(i)(\rho_P(i)-1)\cdots(\rho_P(i)-\rho(i)+1)\frac{\log^{\rho_P(i)}\lambda_i}{\log^{\rho(i)}\lambda_i}.
\end{multline}

The set $R(\rho_P)$ is related to subsequences of $P$.  For $0\leq k\leq r$, let $S(P,k)$ denote the set of subsequences of $P$ of length $k$.  For $Q\in S(P,k)$ we clearly have $\rho_Q\leq \rho_P$.  Conversely, if we are given $\rho\in R(\rho_P)$, then setting $k=\sum_{i=1}^N \rho(i)$, there exists $Q\in S(P,k)$ such that $\rho_Q=\rho$.  Furthermore, it is easy to see that for this $\rho$ we have
\[ {\rm card}\{Q\in S(P,k)\mid \rho_Q=\rho\} = \prod_{i=1}^N \rho_P(i)(\rho_P(i)-1)\cdots(\rho_P(i)-\rho(i)+1). \]
We can thus rewrite (7.8) as
\begin{multline}
\Psi^P_u(\lambda) = \sum_{\substack{k_1,\dots,k_N\in{\mathbb Z}\\ \sum k_i{\bf a}_i = u}} \lambda_1^{k_1}\cdots \lambda_N^{k_N} \cdot \\
\sum_{k=0}^r \sum_{Q=(p_{j_1},\dots,p_{j_k})\in S(P,k)} \bigg(\prod_{i=1}^N M_{k_i,\rho_Q(i)}\bigg) \frac{\log\lambda_{p_1}\cdots\log\lambda_{p_r}} {\log\lambda_{p_{j_1}}\cdots\log\lambda_{p_{j_k}}}.
\end{multline}

For $Q\in S(P,k)$, we set
\begin{equation}
\Phi^Q_u(\lambda) = \sum_{\substack{k_1,\dots,k_N\in{\mathbb Z}\\ \sum k_i{\bf a}_i = u}} \bigg(\prod_{i=1}^N M_{k_i,\rho_Q(i)}\bigg)\lambda_1^{k_1}\cdots \lambda_N^{k_N},
\end{equation}
giving finally
\begin{equation}
\Psi^P_u(\lambda) = \sum_{k=0}^r \sum_{Q=(p_{j_1},\dots,p_{j_k})\in S(P,k)} \Phi^Q_u(\lambda)\frac{\log\lambda_{p_1}\cdots\log\lambda_{p_r}} {\log\lambda_{p_{j_1}}\cdots\log\lambda_{p_{j_k}}}.
\end{equation}
Note that $\Phi^Q_u(\lambda)$ depends only on $\rho_Q$, hence is independent of the original sequence $P$ for which $S(P,k)$ contains $Q$.  

This expression $\Psi^P_u(\lambda)$ satisfies the box operators by Proposition 7.6 and the series $\Phi_u^Q(\lambda)$ appearing in (7.11) satisfy the Euler operators for the parameter $u$ by~(7.10).  We say that $\Psi^P_u(\lambda)$ is an
{\it $r$-th order quasisolution for the parameter $u$}.  We have shown how to associate an $r$-th order quasisolution $\Psi_u^P(\lambda)$ for the parameter $u$ to each sequence $P=(p_1,\dots,p_r)\in{\mathcal P}_r$.  We call the collection $\{\Psi_u^P(\lambda)\}_{P\in{\mathcal P}_r}$ a {\it complete set of $r$-th order quasisolutions for the parameter $u$}.

\begin{proposition}
Let $\{\Psi_u^P(\lambda)\}_{P\in{\mathcal P}_r}$ be a complete set of $r$-th order quasisolutions for the parameter $u$ and let $l^{(k)}=(l^{(k)}_1,\dots,l^{(k)}_N)\in L$ for $k=1,\dots,r$.  Then the expression 
\begin{equation}
\sum_{P=(p_1,\dots,p_r)\in{\mathcal P}_r} l^{(1)}_{p_1}\cdots l^{(r)}_{p_r} \Psi^P_u(\lambda) 
\end{equation}
is a solution of the $A$-hypergeometric system with parameter $u$.
\end{proposition}

\begin{proof}
The expression (7.13) satisfies the box operators because it is a linear combination of expressions $\Psi^P_u(\lambda)$ that satisfy the box operators.  We need to prove that it satisfies the Euler operators with parameter~$u$.

Substitute (7.11) into (7.13):
\begin{equation}
\sum_{P=(p_1,\dots,p_r)\in{\mathcal P}_r} l^{(1)}_{p_1}\cdots l^{(r)}_{p_r}
\sum_{k=0}^r \sum_{Q=(p_{j_1},\dots,p_{j_k})\in S(P,k)} \Phi^Q_u(\lambda)\frac{\log\lambda_{p_1}\cdots\log\lambda_{p_r}}{{\log\lambda}_{p_{j_1}}\cdots{\log\lambda}_{p_{j_k}}}.
\end{equation}
Now reverse the order of summation: For a sequence $Q=(p_{j_1},\dots,p_{j_k})\in{\mathcal P}_k$, the set of all sequences $P=(p_1,\dots,p_r)\in{\mathcal P}_r$ for which $Q\in S(P,k)$ has cardinality $N^{r-k}$: the values $p_l$ for $l\not\in\{{j_1},\dots,{j_k}\}$ can be assigned arbitrarily from $\{1,\dots,N\}$.  Expression (7.14) thus equals
\begin{multline}
\sum_{k=0}^r \sum_{1\leq j_1<\dots<j_k\leq r}\sum_{{p_{j_1}},\dots,{p_{j_k}} = 1}^N l^{(j_1)}_{p_{j_1}}\cdots l^{(j_k)}_{p_{j_k}}\Phi^{(p_{j_1},\dots,p_{j_k})}_u(\lambda)\cdot \\
%\sum_{i_1,\dots,i_r=1}^N \frac{l^{(1)}_{i_1}\log\lambda_{i_1}\cdots l^{(r)}_{i_r}\log\lambda_{i_r}}
%{l^{(j_1)}_{i_{j_1}}\log\lambda_{i_{j_1}}\cdots l^{(j_k)}_{i_{j_k}}\log\lambda_{i_{j_k}}}.
\sum_{p_1,\dots\hat{p}_{j_1},\dots,\hat{p}_{j_k}\dots,p_r=1}^N l^{(1)}_{p_1}\cdots\widehat{l^{(j_1)}_{p_{j_1}}}\cdots\widehat{l^{(j_k)}_{p_{j_k}}}\cdots l^{(r)}_{p_r}\cdot \\ \log\lambda_{p_1}\cdots\widehat{\log\lambda}_{p_{j_1}}\cdots\widehat{\log\lambda}_{p_{j_k}} \cdots\log\lambda_{p_r}.
\end{multline}
Note that the innermost sum in (7.15) equals
\[ \prod_{i=1,\dots,\hat{\jmath}_1,\dots,\hat{\jmath}_k,\dots,r} \log\lambda^{l^{(i)}}, \]
%\[ \prod_{\substack{i=1,\dots,r\\ i\neq j_1,\dots,j_k}} \log\lambda^{l^{(i)}}, \]
so (7.15) simplifies to 
\begin{equation}
\sum_{k=0}^r \sum_{1\leq j_1<\dots<j_k\leq r}\sum_{{p_{j_1}},\dots,{p_{j_k}} = 1}^N l^{(j_1)}_{p_{j_1}}\cdots l^{(j_k)}_{p_{j_k}}\Phi^{(p_{j_1},\dots,p_{j_k})}_u(\lambda) \prod_{\substack{i=1,\dots,r\\ i\neq j_1,\dots,j_k}} \log\lambda^{l^{(i)}}. 
\end{equation}
Every monomial $\lambda^v$ appearing in some $\Phi^Q_u(\lambda)$ satisfies the Euler operators with parameter $u$, which implies by a straightforward calculation that $\lambda^v\prod_{\substack{i=1,\dots,r\\ i\neq j_1,\dots,j_k}} \log\lambda^{l^{(i)}}$ also satisfies the Euler operators with parameter $u$.  It follows that the expression (7.16) (and hence (7.13), which it equals) satisfies the Euler operators with parameter $u$.
\end{proof}

Let $l^{(k)} = (l_1^{(k)},\dots,l^{(k)}_N)\in L$ for $k=1,\dots,r$.  Let the $\Psi^P_u(\lambda)$ be as in (7.11) and let $\Psi^P(\lambda,T)$ be as in (7.4).  For $u\in M$, put
\begin{equation}
A_u = \sum_{P=(p_1,\dots,p_r)\in{\mathcal P}_r} l^{(1)}_{p_1}\cdots l^{(r)}_{p_r} \Psi^P_{-u}(\lambda) 
\end{equation}
It follows from (7.5) that $\partial_jA_u = A_{u+{\bf a}_j}$ and it follows from Proposition~7.12 that $A_u$ satisfies the Euler operators with parameter $-u$.  We therefore get the following result.
\begin{proposition}
Let $A_u$ be given by (7.17) for $u\in M$.  If all $A_u$ lie in some ${\mathcal D}$-module ${\mathcal F}$, then the series $\sum_{u\in M} A_ux^{-u}$ lies in ${\mathcal K}({\mathcal F})$.
\end{proposition}

Again, we apply Corollary 3.7 to get the following result.
\begin{corollary}
Suppose that $m=2n+1$ and that all $A_u$ for $u\in M$ lie in some ${\mathcal D}$-module ${\mathcal F}$.  If $n\Delta(B)$ contains no interior lattice points, then $A_u$ is algebraic over ${\mathbb C}(\lambda)$ for $u\in M^\circ$.
\end{corollary}

{\bf Remark.}  Of course, under the hypothesis of Corollary 7.19, no logarithms will appear in $A_u$ for $u\in M^\circ$ (see the examples in the next section).

\section{Logarithmic examples}  

In this section, we explain how the series (1.2) arise from the logarithmic solutions described in Proposition~7.18.
For this section only, we reprise the notation from the beginning of Section~1.

Let $\alpha_1,\dots,\alpha_n,\beta_{n+1},\dots,\beta_m$ be a sequence of positive integers satisfying
\[ \sum_{i=1}^n \alpha_i = \sum_{j=n+1}^m \beta_j. \]
Throughout this section we deal with a fixed set $B$, so the slight changes we make in notation should not cause confusion.  We take $B=\{{\bf b}_0,\dots,{\bf b}_{m+1}\}\subseteq{\mathbb R}^m$ to be the following set of vectors: ${\bf b}_0 = {\bf 0}$, ${\bf b}_1,\dots,{\bf b}_m$ are the standard unit basis vectors in ${\mathbb R}^m$, and
\[ {\bf b}_{m+1} = (\alpha_1,\dots,\alpha_n,-\beta_{n+1},\dots,-\beta_m). \]  
We denote the convex hull of these points by $\Delta(\alpha,\beta)\subseteq{\mathbb R}^m$ rather than by $\Delta(B)$ as we did previously.  Put ${\bf a}_i = ({\bf b}_i,1)\in{\mathbb R}^{m+1}$ and $A=\{{\bf a}_i\}_{i=0}^{m+1}$.  We denote by $\Delta_0(\alpha,\beta)\subseteq{\mathbb R}^{m+1}$ the convex hull of $A\cup\{{\bf 0}\}$, rather than by $\Delta_0(A)$ as we did previously.  For the lattice of relations $L$ on the set $A$ we have
\begin{align*}
L & = \bigg\{ (l_0,\dots,l_{m+1})\in{\mathbb Z}^{m+2} \mid \sum_{i=0}^{m+1} l_i{\bf a}_i = {\bf 0}\bigg\} \\
 & = \{ l(-1,-\alpha_1,\dots,-\alpha_n,\beta_{n+1},\dots,\beta_m,1)\mid l\in{\mathbb Z}\}.
\end{align*}
Put $\gamma = (-1,-\alpha_1,\dots,-\alpha_n,\beta_{n+1},\dots,\beta_m,1)$ a generator of $L$.  

Put 
\[ u^{(0)} = \sum_{i=0}^n {\bf a}_i = (1,\dots,1,0,\dots,0,n+1) \]
($1$ repeated $n$ times, $0$ repeated $m-n$ times).  If we set $v_i=-1$ for $i=0,\dots,n$ and $v_i=0$ for $i=n+1,\dots,m+1$, then 
\begin{equation}
\sum_{i=0}^{m+1} v_i{\bf a}_i = -u^{(0)},
\end{equation}
\[ L_v = \{ l(-1,-\alpha_1,\dots,-\alpha_n,\beta_{n+1},\dots,\beta_m,1)\mid l\in{\mathbb Z}_{\geq 0}\}, \]
and the series $\Psi_v(\lambda)$ of (6.1) becomes
\begin{equation}
\Psi_v(\lambda) = 
\bigg(\prod_{i=0}^n \lambda_i\bigg)^{-1}\sum_{l=0}^\infty [-1]_{-l} \bigg(\prod_{i=1}^n [-1]_{-l\alpha_i}\bigg)\bigg(\prod_{j=n+1}^{m} [0]_{l\beta_j}\bigg)[0]_l\lambda^{l\gamma},
\end{equation}
or, using the definition of the symbol $[z]_k$,
\begin{equation}
 \Psi_v(\lambda) = 
\bigg(\prod_{i=0}^n \lambda_i\bigg)^{-1}\sum_{l=0}^\infty (-1)^{l(1+\sum_{i=1}^n \alpha_i)}
\frac{\prod_{i=1}^n (\alpha_il)!}{\prod_{j=n+1}^m (\beta_jl)!}\,\lambda^{l\gamma}.
\end{equation}
Our explicit description of $L$ shows that $v$ has minimal negative support, so $\Psi_v(\lambda)$ is a solution of the $A$-hypergeometric system with parameter $-u^{(0)}$.  

\begin{proof}[Proof of Proposition 1.3]
Note that the polytope $n\Delta(\alpha,\beta)$ contains no interior lattice points if and only if the polytope $(n+1)\Delta_0(\alpha,\beta)$ contains no interior lattice points.  Consider the coordinate change on ${\mathbb R}^{m+1}$ that sends $(u_1,\dots,u_m,u_{m+1})$ to $(u_1,\dots,u_m, u_{m+1}-u_1-\cdots-u_m)$.  This transforms the polytope $\Delta_0(\alpha,\beta)$ to the polytope $\tilde{\Delta}_0(\alpha,\beta)$ having vertices at the origin, the standard unit basis vectors, and the point $(\alpha_1,\dots,\alpha_n,-\beta_{n+1},\dots,-\beta_m,1)$.  And since this transformation is unimodular, the polytope $(n+1)\Delta_0(\alpha,\beta)$ contains no interior lattice points if and only if the polytope $(n+1)\tilde{\Delta}_0(\alpha,\beta)$ contains no interior lattice points.  The assertion of the proposition now follows from the case $r=1$ of \cite[Theorems~1.7 and~1.12(a)]{AS5}.
\end{proof}

It follows from \cite[Corollary 2.22]{AS5} that, aside from the trivial case where the $\alpha_i$ are a rearrangement of the $\beta_j$, the series (1.2) cannot have integral coefficients unless $m>2n$. 
\begin{proposition}
Suppose that $m>2n$ and fix $r\leq n+1$.  Then there exists a ${\mathcal D}$-module ${\mathcal F}$ that contains for all $u\in{\mathbb Z}^{m+1}$ the $r$-th order logarithmic solutions $(7.13)$ of the $A$-hypergeometric system with parameter $u$.
\end{proposition}

\begin{proof}
Fix $r\leq n+1$ and $u\in{\mathbb Z}^{m+1}$.  We need to show that the exponents $k$ of the monomials $\lambda^k$ that appear in (7.8) with a nonzero coefficient all lie in a translate of a certain cone with vertex at the origin.  The translate may depend on $u$, but the cone itself must be independent of~$u$.  

Let $P\in{\mathcal P}_r$ and put $Q=\{n+1,\dots,m+1\}$.  The set $Q$ has cardinality $m+1-n>n+1$.  Since $r\leq n+1$, there exists $q\in Q$ such that $\rho_P(q) = 0$.  This implies that $\rho(q) = 0$ for all $\rho\in R(\rho_P)$.  Since $M_{k_q,0}=0$ if $k_q<0$, Equation (7.8) implies that the coefficient of $\lambda^k$ in $\Psi^P_u(\lambda)$ equals $0$ if $k_q<0$.  We therefore examine when $k_q<0$ in (7.8).

Suppose that $k^{(0)} = (k^{(0)}_0,\dots,k^{(0)}_{m+1})\in{\mathbb Z}^{m+2}$ satisfies $\sum_{i=0}^{m+1} k^{(0)}_i{\bf a}_i = u$.  If $k=(k_0,\dots,k_{m+1})\in{\mathbb Z}^{m+2}$ also satisfies $\sum_{i=0}^{m+1} k_i{\bf a}_i=u$, then $k-k^{(0)}\in L$, so we may write
\begin{equation}
k = k^{(0)} + l(-1,-\alpha_1,\dots,-\alpha_n,\beta_{n+1},\dots,\beta_m,1)
\end{equation}
for some $l\in {\mathbb Z}$.  Equation (8.5) implies that for all $q\in Q$, $k_q<0$ for all but finitely many $l<0$.  It follows that if the coefficient of $\lambda^k$ in $\Psi^P_u(\lambda)$ in Eq.~(7.8) is nonzero, then $k$ lies in a translate of the real cone
\begin{equation}
\{ r(-1,-\alpha_1,\dots,-\alpha_n,\beta_{n+1},\dots,\beta_m,1)\mid r\in {\mathbb R}_{\geq 0}\}.
\end{equation}
\end{proof}

We record for later use an observation made during the proof of Proposition 8.4.
\begin{lemma}
If there is an index $i\in\{0,1,\dots,m+1\}$ for which $\rho_P(i)=0$ and $k_i<0$, then the coefficient of $\lambda^k$ in (7.8) equals $0$.
\end{lemma}

The series (8.3) will appear in the $(n+1)$-st order logarithmic solution.  We begin with some lemmas.

A short calculation determines the inequalities defining the polytope~$\Delta(\alpha,\beta)$.  One has $(x_1,\dots,x_m)\in\Delta(\alpha,\beta)$ if and only if
\begin{equation}
x_1+\cdots+x_m\leq 1 + \min \bigg\{ 0,\frac{x_{n+1}}{\beta_{n+1}},\dots,\frac{x_m}{\beta_m} \bigg\}
\end{equation}
and
\begin{equation}
 x_i\geq -\alpha_i\min\bigg\{ 0,\frac{x_{n+1}}{\beta_{n+1}},\dots,\frac{x_m}{\beta_m} \bigg\}\quad\text{for $i=1,\dots,n$.}
\end{equation}
\begin{comment}
From these inequalities, we see that 
\[ {\bf b}_0 + {\bf b}_1+\cdots+{\bf b}_n = (1,\dots,1,0,\dots,0) \]
($1$ repeated $n$ times followed by $0$ repeated $m-n$ times) is an interior point of $(n+1)\Delta(\alpha,\beta)$, hence $u^{(0)}$ is an interior point of $C(A)$.  
\end{comment}
A further calculation using (8.8) and~(8.9) establishes the following result.
\begin{lemma}
Suppose $m>2n$ and let $Q\subseteq\{0,\dots,m+1\}$ be a subset of cardinality $\leq n+1$.  If $Q\neq \{0,\dots,n\}$, then the ${\bf b}_q$ for $q\in Q$ all lie on a common face of $\Delta(\alpha,\beta)$.
\end{lemma}

We apply this lemma to get information about interior lattice points of $C(A)$.  
\begin{lemma}
Suppose that $\sum_{i=0}^{m+1} k_i{\bf a}_i$ (where the $k_i$ are integers) is an interior lattice point of $C(A)$.  Set
\[ Q=\{ q\in\{0,\dots,m+1\}\mid k_q>0\}. \]
If $m>2n$, then $Q=\{0,1,\dots,n\}$ or ${\rm card}(Q)>n+1$.
\end{lemma}

\begin{proof}
Since $\sum_{i=0}^{m+1} k_i{\bf a}_i$ is an interior lattice point of $C(A)$, so is
\[ \sum_{i=0}^{m+1}k_i{\bf a}_i + \sum_{j\not\in Q} (-k_j){\bf a}_j = \sum_{q\in Q} k_q{\bf a}_q \]
(since $-k_j\geq 0$ for $j\not\in Q$).  Since $\sum_{q\in Q}k_q{\bf a}_q$ is an interior lattice point of~$C(A)$, so is $\sum_{q\in Q} {\bf a}_q$.  But by Lemma~8.10, if ${\rm card}(Q)\leq n+1$ and $Q\neq \{0,1,\dots,n\}$, then the ${\bf a}_q$ for all $q\in Q$ lie on a common face of~$C(A)$, in which case $\sum_{q\in Q} {\bf a}_q$ would not be an interior lattice point of $C(A)$.
\end{proof}

We now consider the $(n+1)$-st order quasisolutions of the $A$-hyper\-ge\-o\-met\-ric system described in Proposition 7.12.  We use Lemma~8.11 to describe the series $\Psi_{-u}^P(\lambda)$ appearing in (7.17) for $u\in M^\circ$ when $P\in{\mathcal P}_{n+1}$.  

\begin{lemma}
Suppose that $m>2n$.  Let $P\in{\mathcal P}_{n+1}$ and suppose that $P$ is not a rearrangement of the sequence $(0,1,\dots,n)$.  Then
$\Psi^P_{-u}(\lambda)=0$ for all $u\in M^\circ$, where $\Psi^P_{-u}(\lambda)$ is given in (7.8).
\end{lemma}

\begin{proof}
Let $u\in M^\circ$.  By (7.8), the series $\Psi_{-u}^P(\lambda)$ is a sum over $k_0,\dots,k_{m+1}\in {\mathbb Z}$ satisfying $\sum_{i=0}^{m+1} k_i{\bf a}_i = -u$.  By Lemma~8.7, we can restrict this sum to monomials $\lambda^k$ satisfying
$k_i<0$ only if $\rho_P(i)\neq 0$.  Rewriting this equation as
\begin{equation}
 \sum_{i=0}^{m+1} (-k_i){\bf a}_i = u,
\end{equation}
we have $-k_i>0$ for at most $n+1$ values of $i$ since $P\in{\mathcal P}_{n+1}$.  Lemma~8.11 implies that (8.13) is impossible unless $P$ is, up to reordering, the sequence $(0,1,\dots,n)$.
\end{proof}

Lemma 8.12 simplifies the $(n+1)$-st order logarithmic solutions described in Proposition~7.12 when $u\in M^\circ$.
\begin{corollary}
If $m>2n$ and $u\in M^\circ$, then $\Psi_{-u}^{(0,1,\dots,n)}(\lambda)$ is a solution of the $A$-hypergeometric system with parameter $-u$.
\end{corollary}

We can now apply Proposition 8.4 and Corollary 7.19 to conclude the following.
\begin{corollary}
Suppose that $m=2n+1$ and $n\Delta(\alpha,\beta)$ contains no interior lattice points.  Then $\Psi_{-u}^{(0,1,\dots,n)}(\lambda)$ is algebraic over ${\mathbb C}(\lambda)$ for all $u\in M^\circ$.
\end{corollary}

Put $P_0 = (0,1,\dots,n)$.  To complete this section, we give the explicit formula for $\Psi_{-u}^{P_0}(\lambda)$ when $u\in M^\circ$ by examining the coefficients of $\lambda^k$ in Eq.~(7.8).

Fix $k=(k_1,\dots,k_N)\in{\mathbb Z}^N$ satisfying
\[ \sum_{i=1}^N k_i{\bf a}_i = -u, \]
i.~e.,
\[ \sum_{i=1}^N (-k_i){\bf a}_i = u\in M^\circ. \]
By Lemma 8.11 there are two possibilities: either $-k_i>0$ for $i=0,\dots,n$ and $-k_i\leq 0$ for $i>n$ or
\[ {\rm card}\{i\mid -k_i>0\}>n+1. \]
In the latter case we must have $-k_i>0$ for some $i>n$, i.~e., $k_i<0$ for some $i$ with $\rho_{P_0}(i) = 0$.  Lemma~8.7 then implies that the coefficient of $\lambda^k$ equals~0.

We may therefore assume that we are in the case where $-k_i>0$ for $i=0,\dots,n$ and $-k_i\leq 0$ for $i>n$.  We now examine the sum over $\rho\in R(\rho_{P_0})$ that appears in~(7.8).  

We have $\rho_{P_0}(i)=1$ for $i\in P_0$ and $\rho_{P_0}(i) = 0$ for $i\not\in P_0$.  So if $\rho\in R(\rho_{P_0})$, then $\rho(i)=1$ or $\rho(i)=0$ for $i\in P_0$ and $\rho(i)=0$ for $i\not\in P_0$.  Since $M_{k_i,0}=0$ for $k_i<0$ and $k_i<0$ for $i=0,\dots,n$, the only element $\rho$ of $R(\rho_{P_0})$ that makes a nonzero contribution to the coefficient of $\lambda^k$ in Eq.~(7.8) is $\rho=\rho_{P_0}$.  That contribution is
\[ \prod_{i=0}^{m+1} M_{k_i,\rho_{P_0}(i)} = \prod_{i=0}^n M_{k_i,1} \cdot \prod_{j=n+1}^{m+1} M_{k_j,0} \]
(in particular, no logarithms appear).
From the definition of the $M_{k,i}$, this equals
\[ (-1)^{n+1 - \sum_{i=0}^n k_i}\frac{\prod_{i=0}^n (-k_i-1)!}{\prod_{j=n+1}^{m+1}k_j!}, \]
hence
\begin{equation}
\Psi_{-u}^{P_0}(\lambda) = 
\sum_{\substack{k_0,k_1,\dots,k_n\in{\mathbb Z}_{<0}\\ k_{n+1},\dots,k_{m+1}\in{\mathbb Z}_{\geq 0}\\ \sum_{i=0}^{m+1} k_i{\bf a}_i = -u}} (-1)^{n+1 - \sum_{i=0}^n k_i}\frac{\prod_{i=0}^n (-k_i-1)!}{\prod_{j=n+1}^{m+1}k_j!}\lambda^k.
\end{equation}

We may then restate Corollary 8.15 in an explicit form.
\begin{corollary}
Suppose that $m=2n+1$ and $n\Delta(\alpha,\beta)$ contains no interior lattice points.  Then the series (8.16) is algebraic over ${\mathbb C}(\lambda)$ for all $u\in M^\circ$.
\end{corollary}

{\bf Example.}  The series (1.2) can be recovered as a special case of (8.16).  Take
\[ u^{(0)}=\sum_{i=0}^n {\bf a}_i = (1,\dots,1,0,\dots,0,n+1) \]
($1$ repeated $n$ times, $0$ repeated $m-n$ times).  Equations (8.8) and (8.9) imply that $u^{(0)}$ is an interior lattice point of $C(A)$.  One checks that $(k_0,\dots,k_{m+1})$ satisfies the conditions of the summation on the right-hand side of (8.16) if and only if it lies in the set
\[ \{(-l-1,-l\alpha_1-1,\dots,-l\alpha_n-1,l\beta_{n+1},\dots,l\beta_{m+1},l)\mid l\in{\mathbb Z}_{\geq 0}\}. \]
For $u=u^{(0)}$, the series (8.16) thus becomes the series (8.3).  Corollary~8.17 now implies that the series (1.2) is an algebraic function when $m=2n+1$ and $n\Delta(\alpha,\beta)$ contains no interior lattice points, the observation that originally motivated this work.

{\bf Remark.}  Corollary 8.14 implies that if $m>2n$ and $u\in M^\circ$, then the series (8.16) is a solution of the $A$-hypergeometric system with parameter $-u$.  We conjecture that if in addition $n\Delta(\alpha,\beta)$ contains no interior lattice points, then the series (8.16) has integral coefficients.  Proposition~1.3 shows that this assertion holds for the special case $u=u^{(0)}$ of the preceding example that gives the series (8.3).

\section{Appendix: Algebraic Solutions, \rm by Nicholas M. Katz}

The $X/\C(\lambda)$ of the paper is the generic fibre of the family $\X$ of hypersurfaces  $\sum_i \lambda_i x^{{\bf b}_i}=0$ viewed over
the affine space $\A^N/\C$ of the $\lambda_i$.  This generic fibre is smooth and geometrically connected over $\C(\lambda)$. At the expense of extending scalars from $\C(\lambda)$ to a finite extension field $K/\C(\lambda)$, by Hironaka\cite{H} there exists a projective smooth
$\overline{X}/K$ with geometrically connected fibres, and a normal crossing divisor $D = \cup_j D_j$ in $\overline{X}/K$ such that $X\otimes_{\C(\lambda)}K$ is $\overline{X} \setminus \cup_j D_j$. By the usual ``spreading out" argument, there exists a
dense open set $U \subset \A^N/\C$ over which $\X$ is smooth, say $$ f:\X_U \rightarrow U,$$ and
over which each de Rham cohomology sheaf $H^i_{DR}(\X/U)$
is locally free of finite rank on $U$, and there exists a finite etale connected covering $\E \rightarrow U$ with $K$ the function field of $\E$, and a projective smooth $\overline{\X}/\E$ with a normal crossing divisor $\D = \cup_j \D_j$ over $\E$ such that, over $\E$,  $\X_{\E}:=\X_{U}\otimes\E$ is $\overline{\X} \setminus 
 \cup_j \D_j$. 

In this setting, for each integer $i \geq 0$, the de Rham cohomology sheaf $H^i_{DR}(\X_{\E}/\E)$ is the de Rham incarnation of a mixed Hodge structure on $\E$, whose weight filtration begins as
 $$W_{i-1}H^i_{DR}(\X_{\E}/\E) = 0,$$
  $$W_iH^i_{DR}(\X_{\E}/\E) = {\rm the \ image \ of \ } H^i_{DR}(\overline{\X}/\E).$$
In particular, 
$$W_iH^i_{DR}(\X_{\E}/\E) ={\rm gr}_i^W(H^i_{DR}(\X_{\E}/\E) )$$
is the de Rham incarnation of a polarized variation of pure Hodge structure.

\begin{theorem} Suppose that $i$ is even, $i=2n$, and suppose the Hodge filtration on $W_{2n}H^{2n}_{DR}(\X_{\E}/\E)$ is given by
$$F^i  =W_{2n}H^{2n}_{DR}(\X_{\E}/\E) {\rm \  for \ } i \leq n, F^i  = 0 {\rm \  for \ }i > n.$$
Then  $W_{2n}H^{2n}_{DR}(\X_{\E}/\E)$ is the de Rham incarnation of a variation of pure Hodge structure which becomes constant on a finite etale covering of $\E$. 
\end{theorem}

\begin{proof}This is \cite[Proposition~4.2.1.3]{K}.
\end{proof}

\begin{corollary} On a finite etale connected covering $\E_1/\E$, $W_{2n}H^{2n}_{DR}(\X_{\E}/\E)$ with its Gauss-Manin connection becomes isomorphic to the trivial $\sD$-module $(\sO_{\E_1}^r, d)$, for $r := \rank \,W_{2n}H^{2n}_{DR}(\X_{\E}/\E)$.
\end{corollary}

\begin{corollary} Over a finite extension field $L/\C(\lambda)$, the Gauss-Manin connection on $W_{2n}H^{2n}_{DR}(X/\C(\lambda))$ becomes isomorphic to the  trivial $\sD$-module $(\C(\lambda)^r, d)$, for $r:=\dim W_{2n}H^{2n}_{DR}(X/\C(\lambda)$. In down to earth terms, 
all solutions of the Picard-Fuchs equations for  $W_{2n}H^{2n}_{DR}(X/\C(\lambda))$ with its Gauss-Manin connection are algebraic functions of the $\lambda$'s.
\end{corollary}


\begin{thebibliography}{99}

\bibitem{AS2} A. Adolphson and S. Sperber.  On twisted de Rham cohomology.  Nagoya Math.\ J. {\bf 146} (1997), 55--81.

\bibitem{AS1} A. Adolphson and S. Sperber.  Dwork cohomology, de Rham cohomology, and hypergeometric functions.  Amer.\ J. Math.\ {\bf 122} (2000), no.\ 2, 319--348.

\bibitem{AS} A. Adolphson and S. Sperber.  On the $p$-integrality of $A$-hypergeometric series.  Available at 
{\tt arXiv:1311.5252}.

\bibitem{AS3} A. Adolphson and S. Sperber.  On logarithmic solutions of $A$-hypergeometric systems.  Available at {\tt arXiv:1402.5173}.

\bibitem{AS5} A. Adolphson and S. Sperber.  On the integrality of factorial ratios and mirror maps.  Available at {\tt arXiv:1802.08348}

\bibitem{BD} F. Baldassarri and B. Dwork.  On second order linear differential equations with algebraic solutions. Amer.\ J. Math.\ {\bf 101}  (1979), no.\ 1, 42--76. 

\bibitem{B} V. Batyrev.  Variations of the mixed Hodge structure of affine hypersurfaces in algebraic tori.  Duke Math.\ J. {\bf 69} (1993), no.\ 2, 349--409.

\bibitem{Be} F. Beukers.  Algebraic $A$-hypergeometric functions. Invent.\ Math.\  {\bf 180}  (2010), no.\ 3, 589--610. 

\bibitem{BH} F. Beukers and G. Heckman.  Monodromy for the hypergeometric function $_nF_{n-1}$. Invent.\ Math.\ {\bf 95}  (1989), no.\ 2, 325--354. 

\bibitem{Bo} J. Bober.  Factorial ratios, hypergeometric series, and a family of step functions. J. Lond.\ Math.\ Soc.\ (2)  {\bf 79} (2009), no.\ 2, 422--444. 

\bibitem{De-HodgeII} P. Deligne.  Th\'{e}orie de Hodge. II. Inst.\ Hautes \'{E}tudes Sci.\ Publ.\ Math.\ No.\ 40 (1971), 5--57. 

\bibitem{D} B. Dwork.  On the zeta function of a hypersurface. II. Ann.\ of Math.\ (2)  {\bf 80}  (1964), 227--299.

\bibitem{H} H. Hironaka.  Resolution of singularities of an algebraic variety over a field of characteristic zero. I, II. Ann.\ of Math.\ (2) {\bf 79} (1964), 109--203; ibid.\ (2) {\bf 79} (1964), 205--326.

\bibitem{K} N. Katz.  Algebraic solutions of differential equations ($p$-curvature and the Hodge filtration).  Invent.\ Math.\  {\bf 18}  (1972), 1--118. 

\bibitem{RV} F. Rodriguez-Villegas.  Integral ratios of factorials and algebraic hypergeometric functions.  Available at {\tt arXiv:math/0701362}.

\bibitem{SST} M. Saito, B. Sturmfels, and N. Takayama. Gr\"obner deformations of hypergeometric differential equations. Algorithms and Computation in Mathematics,~6. Springer-Verlag, Berlin, 2000.

\bibitem{S} H. Schwarz.  Ueber diejenigen F\"alle, in welchen die Gaussichen hypergeometrische Reihe eine algebraische Function ihres vierten Elementes darstellt.  Journal f.\ d. reine und angew.\ Math.\ {\bf 75} (1873), 292--335. 

\end{thebibliography}
\end{document}